\documentclass[12pt,reqno]{amsart}
\usepackage{amsopn,amssymb,amsthm,mathrsfs,skak,url,mathrsfs,a4wide,dsfont}
\usepackage{thmtools}
\usepackage[mathscr]{euscript}
\newtheorem{thm}{Theorem}[section]

\newtheorem{lemma}[thm]{Lemma}
\newtheorem{proposition}[thm]{Proposition}

\declaretheorem[numbered=no]{Main Theorem}

\theoremstyle{definition}

\newtheorem{remark}[thm]{Remark}

\newcounter{smallromans}

\newenvironment{romanenumerate}
{\begin{list}{{\normalfont\textrm{(\roman{smallromans})}}}%
  {\usecounter{smallromans}\setlength{\itemindent}{0cm}%
   \setlength{\leftmargin}{5.5ex}\setlength{\labelwidth}{5.5ex}%
   \setlength{\topsep}{.5ex}\setlength{\partopsep}{.5ex}%
   \setlength{\itemsep}{0.1ex}}}%
{\end{list}}

\newcounter{smallromansdash}

{\end{list}}

\newcounter{bigromans} 
  {\end{list}}

\begin{document}
\title[The trace as an average]{The trace as an average over the unit sphere of\\
a normed space with a 1-symmetric basis}
\begin{abstract}We generalise the formula expressing the matrix trace of a given square matrix as the integral of the numerical values of $A$ over the Euclidean sphere to the unit spheres of finite-dimensional normed spaces that have a 1-symmetric basis. Our result is new even in the case of $\ell_p$-norms in $\mathbb{R}^N$ for $p\neq 2$.\end{abstract}

\author[T.~Kania]{Tomasz Kania}
\address{School of Mathematical Sciences, Western Gateway Building, University College Cork, Cork, Ireland}
\email{tomasz.marcin.kania@gmail.com}

\author[K.~E.~Morrison]{Kent E.~Morrison}
\address{American Institute of Mathematics, 600 E.~Brokaw Road, San Jose, CA 95112}
\email{morrison@aimath.org}

\date{\today}
\maketitle

\section{Introduction and the main result}\label{intro}
The fact that the trace ${\rm tr}\,A = a_{11} + a_{22} + \ldots + a_{NN}$ of a given real or complex square matrix $A=[a_{ij}]_{i,j\leqslant N}$ may be expressed as \begin{equation}\label{trfor}{\rm tr}\,A =N\cdot \int\limits_{S} \langle Ax, x\rangle \, \mu({\rm d}x),\end{equation} where $\mu$ is the normalised hypersurface measure on the Euclidean sphere $S$ in $\mathbb{R}^N$ (or $\mathbb{C}^N$), has become a part of the mathematical folklore. Recently, the first-named author presented a proof of the above formula based on algebraic properties of the trace (\cite{Kania}). Equation \eqref{trfor} is a particular case of a more general formula for symmetric 2-tensors on Riemannian manifolds (see, \emph{e.g.}, \cite{lg}).\smallskip

The purpose of this note is to extend formula \eqref{trfor} to the unit spheres of finite-dimensional normed spaces $X$ which are not necessarily manifolds but whose symmetry groups are large enough to accommodate the group of symmetries of a hypercube or of a cross-polytope of the dimension equal to the dimension of $X$. This hypothesis is formalised in the language of Banach space theory---we postpone the explanation of all necessary terminology to Section~\ref{prelim}.\smallskip

Our main result then reads as follows.

\begin{Main Theorem}Let $X$ be a finite-dimensional normed space with a 1-symmetric basis. Let $S_X$ be the unit sphere of $X$ and let $\mu$ be the normalised hypersurface measure on $S_X$. \begin{romanenumerate}\item Let $x_0\in S_X$. Denote by $x_0^*$ the norming functional for $x_0$ in the case where it is unique. Then the function 
$$x\mapsto x^*\quad (x\in S_X)$$
is defined on a subset $D\subseteq S_X$ such that $\mu(S_X\setminus D)=0$ and is continuous on $D$.
\item If $A$ is an $(N\times N)$-matrix, where $N=\dim X$, then
\begin{equation}\label{new}{\rm tr}\,A = N\cdot \int\limits_{S_X} \langle Ax, x^*\rangle\, \mu({\rm d}x).\end{equation}\end{romanenumerate} \end{Main Theorem}
We prove the Main Theorem in Section~\ref{proof}.
\begin{remark}Formally, formula \eqref{new} should read \begin{equation}\label{new2}{\rm tr}\,A = N\cdot \int\limits_{D} \langle Ax, x^*\rangle\, \mu({\rm d}x),\end{equation}
however it does not really matter as $\mu(S_X \setminus D)=0$, hence we understand by $x\mapsto x^*$ ($x\in S_X$)  any measurable extension of $x\mapsto x^*$ ($x\in S_X\setminus D$) to a function $S_X\to S_{X^*}$.\end{remark}\smallskip

\section{Preliminaries and auxiliary results}\label{prelim}
We shall work with real normed spaces however all the results in this note generalise to the complex scalars in a straightforward manner.

\subsection{1-symmetric bases and the hyperoctahedral groups}
Let $X$ be a finite-dimen\-sio\-nal normed space with a normalised basis $(x_n)_{n=1}^N$. We say that the basis $(x_n)_{n=1}^N$ is \emph{1-symmetric} if
$$\left\|\sum_{k=1}^N a_k x_k\right\|_X = \left\|\sum_{k=1}^N \varepsilon_k a_{k} x_{\sigma(k)}\right\|_X$$
for every choice of scalars $a_1, \ldots, a_N$, $\varepsilon_1, \ldots, \varepsilon_N$ with $|\varepsilon_j|=1$ ($j\leqslant N$) and every permutation $\sigma$ of $\{1, \ldots, N\}$. In other words, $(x_n)_{n=1}^N$ is 1-symmetric if the hyperoctahedral group $\mbox{BC}_N$ of degree $N$ acts on the linear combinations of $(x_n)_{n=1}^N$ by isometries. Here, by $\mbox{BC}_N$, we understand the group of $(N\times N)$-matrices that have exactly one non-zero entry in each row and exactly one non-zero entry in each column, and the non-zero entries are either $1$ or $-1$. \smallskip

We remark in passing that the paradigm examples of spaces with a 1-symmetric basis are the finite-dimensional $\ell_p^N$ spaces for $p\in [1,\infty]$. This class of spaces is however much broader and one usually constructs such spaces as ranges of initial basis projections of infinite-dimensional Banach spaces with a 1-symmetric Schauder basis. \smallskip

Let us now prove a lemma concerning the hyperoctahedral group which may be of interest in itself.

\begin{proposition}\label{mainlemma}Let $N\in \mathbb{N}$ and let $A$ be an $(N\times N)$-matrix. Then
$$\sum_{Q\in {\rm BC}_N} Q^TAQ  = \Big((N-1)!\cdot 2^N\cdot {\rm tr}\,A\Big)I_N,$$
where $I_N$ denotes the $N\times N$ identity matrix.\end{proposition}
\begin{proof}Let $e_1, \ldots, e_N$ be the standard basis of $\mathbb{R}^N$. For each $Q\in {\rm BC}_N$ there exists a unique pair $(\sigma^Q, \beta^Q)$, where $\sigma^Q$ is a permutation of $\{1, \ldots, N\}$ and $\beta^Q\colon \{1, \ldots, N\} \to \{-1, 1\}$ such that $$Qe_i = \beta^Q(i)e_{\sigma^Q(i)}\quad (i\leqslant N).$$ Given $Q\in {\rm BC}_N$, it is straightforward to check that the $(i,j)$ entry of $Q^TAQ$ is $$\beta^Q(i)\beta^Q(j)a_{\sigma^Q(i), \sigma^Q(j)}.$$ \smallskip
If $i=j$, then the $(i,i)$ entry of $\sum_{Q\in {\rm BC}_N} Q^TAQ$ is 
\begin{equation*}\begin{split}\sum_{(\sigma^Q,\, \beta^Q)} a_{\sigma(i), \beta(i)} & = \sum_{\sigma^Q} \sum_{\beta^Q} a_{\sigma(i), \sigma(i)} \\
& = \sum_{\sigma \in {\rm BC}_N} 2^N a_{\sigma(i), \sigma(i)} \\
& = \sum_{i=1}^N (N-1)!\cdot 2^N a_{i,i} \\
& = (N-1)! \cdot 2^N \cdot {\rm tr}\,A.\end{split}\end{equation*}
If $i\neq j$, then the $(i,j)$ entry of $\sum_{Q\in {\rm BC}_N} Q^TAQ$ is 0 because 
$$\sum_{(\sigma^Q,\, \beta^Q)} \beta^Q(i)\beta^Q(j)a_{\sigma^Q(i), \sigma^Q(j)} = \sum_{\sigma^Q}\sum_{\beta^Q} \beta^Q(i)\beta^Q(j) a_{\sigma^Q(i), \sigma^Q(j)}, $$
and the inner sums over $\beta^Q$ are all zero.\end{proof}

\subsection{Duality in finite-dimensional normed spaces}
Let $X$ be a finite-dimensional normed space. Then, $X$ is isomorphic as a vector space to $\mathbb{R}^N$ for some $N\in \mathbb{N}$. Henceforth, we shall regard $X$ as $\mathbb{R}^N$ endowed with a norm that comes from $X$ and we shall assume that the canonical basis of $\mathbb{R}^n$ is 1-symmetric with respect to that norm. Since the dual space $X^*$ is $N$-dimensional, we will regard it as a vector space as $\mathbb{R}^N$ too. Then the map $\Phi_{\langle \cdot, \cdot \rangle}\colon X\to X^*$ given by 
$$\Phi_{\langle \cdot, \cdot \rangle}\colon x\mapsto (f\mapsto \langle f,x\rangle\, (f\in X^*)),$$
where $\langle \cdot, \cdot \rangle$ denotes the the canonical inner product in $X\cong \mathbb{R}^N$, is an isomorphism of vector spaces. Under this identification, the duality between $X$ and $X^*$ is given by the scalar product, that is
$$f(x)  = \langle f,x\rangle = \sum_{j=1}^N x_jf_j \quad (x=(x_j)_{j=1}^N\in X, f=(f_j)_{j=1}^N\in X^*).$$
We shall use this identification hereinafter.\smallskip

\subsection{Uniqueness of norming functionals}
Let $X$ be a Banach space, $x_0\in X$ be a non-zero vector and let $f\in X^*$ be a norm-one functional. Then $f$ is \emph{a norming functional for} $x_0$, if $\langle f,x_0\rangle = \|x_0\|_X$. The Hahn--Banach theorem guarantees the existence of at least one norming functional for each non-zero vector in a Banach space. Such functional however need not be unique. We will demonstrate that the set of norm-one elements in a finite-dimensional normed space that have more than one norming functional is null with respect to the hypersurface measure on the unit sphere. Before we state and prove the result, we require two standard facts from analysis.\smallskip

The first one is a well-known characterisation of vectors with unique norming functionals in terms of smoothness of the norm. The proof can be found, \emph{e.g.}, in \cite[pp.~179--180]{Beauzamy}.
\begin{proposition}\label{norming}Let $X$ be a Banach space and let $x_0$ be a non-zero vector in $X$. Then $x_0$ has a unique norming functional if and only if the map $x\mapsto \|x\|_X\; (x\in X)$ is G\^{a}teaux-differentiable at $x_0$.\end{proposition}

The other bit we require is Rademacher's theorem. Let us state it in the version for Riemannian manifolds, which will be convenient for us, since we shall apply it to a function defined on the Euclidean sphere.

\begin{thm}[Rademacher]Let $U$ be a Riemannian manifold and let $F\colon U\to \mathbb{R}^m$ be a locally Lipschitz map. Then $F$ is Fr\'{e}chet-differentiable (hence also G\^{a}teaux-differentiable) almost everywhere with respect to the Riemann--Lebesgue volume measure on $U$.\smallskip

In other words, the set of points in $U$ where $F$ is not differentiable is measurable and has measure zero.\end{thm} 
This more general form of the theorem follows directly from \cite[Theorem 3 on p.~250]{Stein}. We are now in a position to state and prove the main result of this section.

\begin{proposition}\label{zerom}Let $X$ be a finite-dimensional Banach space. Denote by $S_X$ the unit sphere of $X$. Let $\mu$ be the hypersurface measure on $S_X$. Set
$$W=\{x\in S_X\colon x\text{ has more than one norming functional}\}.$$
Then $W$ is $\mu$-measurable and $\mu(W)=0$.
\end{proposition}
\begin{proof}Fix an inner product in $X$ and let $S^{N-1}$ denote the unit sphere with respect to the Euclidean norm that comes from the inner product. Let $F\colon S^{N-1}\to S_X$ be given by
$$F(x) = \frac{x}{\|x\|_X}\quad (x\in S^{N-1}).$$
Since all norms in $X$ are equivalent, $F$ is a bi-Lipschitz homeomorphism. Moreover, $F$ is differentiable at $x_0\in S^{N-1}$ if and only if $x\mapsto \|x\|_X$ is differentiable at $x_0$. By Proposition~\ref{norming}, this is equivalent to the existence of a unique norming functional $x_0^*\in X^*$ for $x_0$. Thus
$$W=\{x\in S_X\colon F^{}\text{ is not differentiable at }F^{-1}(x)\}.$$
By Rademacher's theorem, the hypersurface measure of $F^{-1}(W)$ is zero. Since Lipschitz maps preserve null sets and $F$ is bi-Lipschitz, the set $W=F(F^{-1}(W))$ is $\mu$-measurable and $\mu(W)=0$.\end{proof}

\begin{lemma}\label{invariant}Let $X$ be an $N$-dimensional Banach space ($N\in \mathbb{N}$) with a 1-symmetric basis. Let $Q\in \mbox{BC}_N$ and $x\in X\setminus \{0\}$. If $f$ is a norming functional for $x$, then $Qf$ is a norming functional for $Qx$. Moreover, if $x$ has a unique norming functional, then so has $Qx$. \end{lemma}

\begin{proof}We have $$\langle Qf, Qx \rangle = \langle Q^TQf, x\rangle = \langle f, x \rangle = \|x\|=\|Qx\|,$$where the last equality follows the fact that $Q$ is isometric which is the case in the presence of a $1$-symmetric basis. Now, suppose that $f$ and $g$ are norming functionals for $Qx$. Then $Q^Tf$ and $Q^Tg$ are norming functionals for $x$. As $Q$ is injective, $$\langle Q^Tf, x\rangle= \langle f, Qx\rangle = \|Qx\| = \langle g, Qx\rangle = \langle Q^Tg, x\rangle.$$
Hence $Q^Tf = Q^Tg$, so $f=g$. \end{proof}

\section{Proof of the Main Theorem}\label{proof}
We are now ready to prove the Main Theorem.

\begin{proof}For (i), it follows from Proposition~\ref{zerom} that there exists a set $D\subseteq S_X$ which is $\mu$-measurable, $\mu(S_X\setminus D)=0$ and the assignment $x\mapsto x^*$ is a well-defined function on $D$. By (the proof of) \cite[Proposition~1 on p.~177]{Beauzamy}, the map $x\mapsto x^*$ is continuous on $D$.\smallskip

Let $A$ be an ($N\times N$)-matrix. In order to prove (ii), note that by Lemma~\ref{invariant} and since ${\rm BC}_N$ acts by isometries on $X$, for each $Q\in {\rm BC}_N$ we have
$$\int\limits_{D} \langle A(Qx), (Qx)^*\rangle \, \mu({\rm d}x) = \int\limits_{D} \langle Ax, x^*\rangle \, \mu({\rm d}x)$$
and $$\langle A(Qx), (Qx)^*\rangle = \langle A(Qx), Q(x^*)\rangle = \langle Q^TAQx, x^*\rangle.$$

Summing the integrals over $Q\in {BC}_N$, we have
$$\sum_{Q\in {\rm BC}_N}\; \int\limits_{D} \langle Q^TAQx, x^*\rangle \, \mu({\rm d}x) = N!\cdot 2^N \int\limits_{D} \langle Ax, x^*\rangle \, \mu({\rm d}x),$$
hence
$$\int\limits_{D} \Big\langle \sum_{Q\in {\rm BC}_N} Q^TAQx, x^*\Big\rangle \, \mu({\rm d}x) = N!\cdot 2^N \int\limits_{D} \langle Ax, x^*\rangle \, \mu({\rm d}x).$$
By Proposition~\ref{mainlemma}, 
$$\sum_{Q\in {\rm BC}_N} Q^TAQ  = \Big((N-1)!\cdot 2^N\cdot {\rm tr}\,A\Big)I_N.$$
Consequently,
\begin{equation*}\begin{split}(N-1)!\cdot 2^N\cdot {\rm tr}\,A & = \int\limits_{D} \Big((N-1)!\cdot 2^N\cdot {\rm tr}\,A\Big)\langle x, x^*\rangle\, \mu({\rm d}x)\\
& = N!\cdot 2^N \int\limits_{D} \langle Ax, x^*\rangle \, \mu({\rm d}x),\end{split}\end{equation*}
which proves \eqref{new}.\end{proof}

\subsection*{Closing remarks}
We shall demonstrate that the hypothesis of having a 1-symmetric basis cannot be dropped in the statement of the Main Theorem. To see this, consider $\mathbb{R}^2$ endowed with the norm
$$ \|(x,y)\|_X = \sqrt{x^2 + \frac{y^2}{b^2}}\quad \big((x,y)\in \mathbb{R}^2\big).$$
For each $(x,y)\in S_X$, we have $(x,y)^* = (x, \frac{y}{b^2})$. Let $A=\left[\begin{smallmatrix}1&0\\0 &  0\end{smallmatrix}\right]$. Then, $\langle A(x,y), (x, \frac{y}{b^2})\rangle = x^2$. Let us parametrise the ellipse $S_X$ by $x(t) = \cos t, y(t) = b\sin t$. Then

\begin{equation*}\begin{split}\int\limits_{S_X} x^2\, \mu({\rm d}(x,y)) &  = \frac{\int\limits_0 ^{2 \pi} x(t) ^2 \sqrt {x'(t) ^2 + y' (t) ^2}\, {\rm d} t}{\int\limits_0^{2 \pi} \sqrt{x'(t) ^2 + y' (t) ^2} \,{\rm d t}}  \\
& = \frac{\int \limits _0 ^{2 \pi} \cos ^2 t \sqrt {\sin ^2 t + b^2 \cos ^2 t}\, {\rm d} t}{\int\limits_0^{2\pi} \sqrt{\sin ^2 t + b^2 \cos ^2 t}\, {\rm d} t}  
\end{split}\end{equation*}
When $b=1$ the average is $1/2$, which is $\frac{1}{2} {\rm tr}\, A$ as it should be since the norm is the $\ell_2$-norm. The integrals are continuous in the parameter $b$ and for $b=0$ the average is 
\begin{equation*}
    \frac{4 \int_0^\frac{\pi}{2} \cos^2 t \, \sin t \, {\rm d}t}{4 \int_0^\frac{\pi}{2} \sin t \,{\rm d}t}=\frac{1}{3},
\end{equation*}
and so there are values of $b$ for which the average is not $1/2$.

\end{document}